\documentclass{amsart}
\usepackage{orcidlink}
\usepackage[nobreak]{cite}
\usepackage[utf8]{inputenc}
\usepackage[T1]{fontenc}
\usepackage{hyperref,amssymb,mdwlist,microtype,mathtools,enumitem,color,soul}
\usepackage[capitalize,nameinlink]{cleveref}

\allowdisplaybreaks

\newtheorem{theorem}{Theorem}[section]

\newtheorem{lemma}[theorem]{Lemma}
\newtheorem{prop}[theorem]{Proposition}
\theoremstyle{definition}

\theoremstyle{remark}
\newtheorem{remark}[theorem]{Remark}
\newtheorem{example}[theorem]{Example}
\crefname{equation}{}{}
\Crefname{enumi}{}{}

\DeclareMathOperator{\id}{id}
\DeclareMathOperator{\card}{card}
\DeclareMathOperator{\B}{B}
\DeclareMathOperator{\T}{T}
\newcommand{\BB}{\mathbb{B}}

\newcommand{\Borel}{\mathcal{B}}
\newcommand{\FF}{\mathcal{F}}
\newcommand{\FP}{\mathcal{P}}
\newcommand{\FI}{\mathcal{I}}

\newcommand{\FPS}{\mathcal{P}_{S}}
\newcommand{\FPSp}{\mathcal{P}_{S_p}}
\newcommand{\FPSdwa}{\mathcal{P}_{S_2}}
\newcommand{\FM}{\mathcal{M}}
\newcommand{\FMS}{\mathcal{M}_{S}}
\newcommand{\FMSp}{\mathcal{M}_{S_p}}

\newcommand{\clopen}[2]{\char"5B #1,#2\char"29}
\newcommand{\opencl}[2]{\char"28 #1,#2\char"5D}

\newcommand{\field}[1]{\mathbb{#1}}
\makeatletter
\def\cref@thmoptarg[#1]#2#3#4{%
    \ifhmode\unskip\unskip\par\fi%
    \normalfont%
    \trivlist%
    \let\thmheadnl\relax%
    \let\thm@swap\@gobble%
    \thm@notefont{\fontseries\mddefault\upshape}%
    \thm@headpunct{.}
    \thm@headsep 5\p@ plus\p@ minus\p@\relax%
    \thm@space@setup%
    #2
    \@topsep \thm@preskip               
    \@topsepadd \thm@postskip           
    \def\@tempa{#3}\ifx\@empty\@tempa%
      \def\@tempa{\@oparg{\@begintheorem{#4}{}}[]}%
    \else%
      \refstepcounter[#1]{#3}
      \@namedef{cref@#3@alias}{#1}
      \def\@tempa{\@oparg{\@begintheorem{#4}{\csname the#3\endcsname}}[]}%
    \fi%
    \@tempa}%
\makeatother
\crefname{theorem}{Theorem}{Theorems}
\crefname{lemma}{Lemma}{Lemmas}
\crefname{proposition}{Proposition}{Propositions}
\crefname{prop}{Proposition}{Propositions}
\crefname{corollary}{Corollary}{Corollaries}
\crefname{definition}{Definition}{Definitions}

\begin{document}
\title[Invariant probability measures]{Invariant probability measures under $p$-adic transformations}
\author[O.V.Maslyuchenko]{Oleksandr V. Maslyuchenko \orcidlink{0000-0002-1493-9399}}
\address{Institute of Mathematics\\ University of Silesia\\ Bankowa 14\\ 40-007 Katowice\\ Poland}
\address{Department of Mathematics and Informatics\\ Yuriy Fedkovych Chernivtsi National University\\ Kotsiubynskoho  2\\ 58-012 Chernivtsi\\ Ukraine}
\email{oleksandr.maslyuchenko@us.edu.pl}
\author[J.Morawiec]{Janusz Morawiec \orcidlink{0000-0002-0310-867X}}
\address{Institute of Mathematics\\ University of Silesia\\ Bankowa 14\\ 40-007 Katowice\\ Poland}
\email{janusz.morawiec@us.edu.pl}
\author[T.Zürcher]{Thomas Zürcher \orcidlink{0000-0001-9179-9521}}
\address{Institute of Mathematics\\ University of Silesia\\ Bankowa 14\\ 40-007 Katowice\\ Poland}
\email{thomas.zurcher@us.edu.pl}

\dedicatory{In memory of J\'{a}nos Acz\'{e}l on the occasion of the $100^{th}$ anniversary of his birth}


\begin{abstract}
It is well-known that the Lebesgue measure is the unique absolutely continuous invariant probability measure under the $p$-adic transformation. The purpose of this paper is to characterize the family of all invariant probability measures under the $p$-adic transformation and to provide some description of them. In particular, we describe the subfamily of all atomic invariant measures under the $p$-adic transformation as well as the subfamily of all continuous and singular invariant probability measures under the $p$-adic transformation.
Iterative functional equations play the base role in our considerations.
\end{abstract}
\maketitle


\renewcommand{\theequation}{1.\arabic{equation}}\setcounter{equation}{0}
\section{Introduction}\label{S1} 

The dyadic transformation is an example of the simplest deterministic dynamical system that has so-called chaotic dynamics. It is a special case of the so called $\beta$-transformation, which has been studied by many authors during the last 70 years. The $\beta$-transformation was introduced in~\cite{Renyi1957}, where it was proved that there exists exactly one probability measure that is absolutely continuous (with respect to the Lebesgue measure) and invariant under the $\beta$-transformation. An explicit formula for this invariant probability measure was obtained in~\cite{Gelfond1959} and, independently, in~\cite{Parry1960}, where it was also proved that the $\beta$-transformation is weakly mixing, and hence ergodic. Exactness of the $\beta$-transformation was demonstrated in~\cite{Rohlin1961}.
Additional information on the $\beta$-transformation with extensive literature and broadly discussed topics connected with it can be found in~\cite{Vapstas2024}. 

Results about the existence of the unique absolutely continuous invariant probability measure under the $\beta$-transformation, on its explicit formula and ergodicity behaviour have been motivation for many authors to extend them for more general transformations but still connected with the original $\beta$-transformations (see e.g.\ \cite{TanakaIto1982, DajaniKraaikamp2003, Gora2007, DajaniVries2007, DajaniKalle2010, Kempton2014, KalleMaggioni2022, Suzuki2024}). All the results on the existence of absolutely continuous invariant probability measures are based on a method introduced in~\cite{LasotaYork1973} for deterministic dynamical systems.
These methods were later effectively extended to random dynamical systems (see e.g.\ \cite{Buzzi2000, GoraBoyarski2003, BahsounGora2005, Inoue2012}). 

In many situations, there are also continuous and singular invariant measures or atomic invariant probability measures among the absolutely continuous invariant probability measures under a given transformation.
Moreover, the family of all invariant probability measures under a given transformation can be quite large, however among them, there exists only one invariant probability measure that is absolutely continuous. An example of such a situation is the $p$-adic transformation, which we will consider in this paper. The problem with determining invariant probability measures under a given transformation that are not absolutely continuous is that there is no method for this purpose. To the best of our knowledge, one can get some information on singular invariant probability measures considering the iterated function system associated with the given transformation (if it is even possible) and to show that its attractor is of Lebesgue measure zero. However, by this way we may obtain only very specific singular invariant probability measures.

The purpose of this paper is to describe the family of all invariant probability measures under the $p$-adic transformation. As we know that the family has exactly one absolutely continuous measure (which is exactly the Lebesgue measure) it suffices to determine its subfamily of all atomic measures and its subfamily of all continuous and singular measures. Let us note that large families of continuous and singular invariant probability measures under the dyadic transformation were determined in~\cite{MorawiecZurcher2018, MorawiecZurcher2021, MorawiecZurcher2024}. 

Our main tool in this paper is the functional equation associated with the probability distribution functions of invariant measures that is slightly different from that which arises considering density functions of invariant measures (see e.g.~\cite{Suzuki2019}). 

The main part of this paper is~\cref{S6}, in which we describe the structure of
atomic invariant probability measures under the $p$-adic transformation that are supported on minimal finite orbits. As noted in the introduction of~\cite{DownarowiczHuczek2022} such a description can help in progresses of solving the  Furstenberg $\times 2,\times 3$ conjecture, posed in~\cite{Furstenberg1967}, that is one of the major unsolved problems in the field of ergodic theory in dynamical systems.

\renewcommand{\theequation}{2.\arabic{equation}}\setcounter{equation}{0}
\section{Preliminaries}\label{S2}

Denote by $\Borel$ the $\sigma$-algebra of all Borel subsets of $\clopen{0}{1}$ and by $\FF$ the family of all $\Borel$-measurable maps $\Phi\colon\clopen{0}{1}\to\clopen{0}{1}$. Any measure defined on $\Borel$ is called a \emph{Borel measure} and any $\Phi\in\FF$ is called a \emph{Borel transformation}. 
A Borel probability measure $\mu$ is said to be \emph{invariant under $\Phi\in\FF$} if 
\begin{equation*}
\mu(\Phi^{-1}(B))=\mu(B)\quad\text{for every }B\in\Borel.
\end{equation*}

Throughout this paper, we fix an integer number $p\geq 2$ and consider the \emph{$p$-adic transformation} $S_p\in\FF$ defined by 
\begin{equation*}
S_p(x)=px\hspace*{.5ex}(\hspace*{-2ex}\mod 1).
\end{equation*} 

The aim of this paper is to characterize the set of all invariant measures under $S_p$ and provide some description of them. However, we will formulate some results for more general transformations than $S_p$. For this purpose, we fix numbers $a_1,\ldots,a_{p+1}\in[0,1]$ such that $0=a_1<a_2<\cdots<a_{p+1}=1$. Next, for every $k\in\{1,\ldots,p\}$ we fix an increasing bijection $f_k\colon[a_k,a_{k+1}]\to[0,1]$ and consider $S\in\FF$ defined by
\begin{equation*}
S(x)=f_k(x)\quad\text{ for }x\in\clopen{a_k}{a_{k+1}}\text{ with }k\in\{1,\ldots,p\}.
\end{equation*}
The definition of $S$ does not require defining the functions $f_1,\ldots,f_p$ on the right endpoints of their domains, however it will be important in~\cref{S3} to simplify writing.
Note that in the case where $S$ is  the $p$\nobreakdash-adic transformation~$S_p$, we have $a_{k+1}-a_k=\frac{1}{p}$ and $f_k(x)=px-k+1$ for every $k\in\{1,\ldots,p\}$. 

Denote by $\FM$ the family of all Borel probability measures and put
\begin{equation*}
\FMS=\big\{\mu\in\FM\mid\mu(S^{-1}(B))=\mu(B)\text{ for every }B\in\Borel\big\}.
\end{equation*}

We begin with two simple observations on $\FMS$. 

\begin{remark}\label{rem:Continuity}
Every $\mu\in\FMS$ vanishes on each point of the set $\{a_2,\ldots,a_{p}\}$.
In particular, if $\mu\in\FMSp$, then $\mu(\{\frac{k}{p}\})=0$ for every $k\in\{1,\ldots,p-1\}$.
\end{remark}

Denote by $\delta_x$ the Dirac measure concentrated at the point $x\in\mathbb R$.

\begin{remark}\label{rem:FixedPoints}
Assume that $x\in[0,1]$. Then $\delta_x\in\FMS$ if and only if $S(x)=x$.
In particular, $\delta_x\in\FMSp$ if and only if $x\in\{\frac{k}{p-1}\mid k\in\{0,\ldots,p-2\}\}$. 
\end{remark}


\renewcommand{\theequation}{3.\arabic{equation}}\setcounter{equation}{0}
\section{Invariant measures vs.\ functional equations}\label{S3}

Put
\begin{align*}
\FI&=\left\{\phi\colon[0,1]\to[0,1]\mid \phi \text{ is non-decreasing, }\phi(0)=0,\text{ and } \phi(1)=1\right\},\\
\FP&=\left\{\phi\in\FI\mid\phi \text{ is left-continuous}\right\}.
\end{align*}
Recall (cf.\ e.g.~\cite[Theorem~12.4]{Billingsley1995} or~\cite[9.1.1]{Dudley2002} and note that left- and right-continuous functions correspond to each other) that $\FP$ and $\FM$ are in one-to-one correspondence by the formula
\begin{equation}\label{mu-phi}
\phi(x)=\mu(\clopen{0}{x})\quad\text{ for every }x\in[0,1].
\end{equation}

\begin{prop}\label{prop:FunctionalEquation}
Assume that $\mu\in\FM$ and let $\phi\in\FP$ correspond to $\mu$ by~\cref{mu-phi}. Then $\mu\in\FMS$ if and only if 
\begin{equation}\label{FE}
\phi(x)=\sum_{k=1}^{p}\left[\phi(f_k^{-1}(x))-\phi(f_k^{-1}(0))\right]\quad\text{for every }x\in[0,1].
\end{equation}
\end{prop}

\begin{proof}
($\Rightarrow$) Fix $\mu\in\FMS$ and $x\in[0,1]$. 
Applying~\cref{mu-phi}, we obtain
\begin{align*}
\phi(x)&=\mu(\clopen{0}{x})=\mu\left(S^{-1}(\clopen{0}{x})\right)
=\sum_{k=1}^{p}\mu\left(\left[a_k,f_k^{-1}(x)\right)\right)\\
&=\sum_{k=1}^{p}\left[\mu\left(\left[0,f_k^{-1}(x)\right)\right)
-\mu\left([0,a_k)\right)\right]
=\sum_{k=1}^{p}\left[\phi\left(f_k^{-1}(x)\right)-\phi(a_k)\right]\\
&=\sum_{k=1}^{p}\left[\phi(f_k^{-1}(x))-\phi(f_k^{-1}(0))\right].
\end{align*}

($\Leftarrow$) Fix $\phi\in\mathcal P$ satisfying \cref{FE}. Applying \cref{mu-phi}, for every $y\in[0,1)$, we get 
\begin{equation*}
\mu(S^{-1}([0,y)))=\mu([0,y)).
\end{equation*} 
Hence for all $x,y\in(0,1)$ with $x<y$ we have
\begin{align*}
\mu(S^{-1}(\clopen{x}{y}))&=\mu(S^{-1}([0,y)))-\mu(S^{-1}([0,x)))=\mu([0,y))-\mu([0,x))\\
&=\mu(\clopen{x}{y}).
\end{align*}
Now, it suffices to apply \cite[Theorem~3.3]{Billingsley1995}.
\end{proof}

In the case of the $p$\nobreakdash-adic transformation \cref{FE} takes the form
\begin{equation}\label{FESp}
\phi(x)=\sum_{k=0}^{p-1}\left[\phi\left(\frac{x+k}{p}\right)
-\phi\left(\frac{k}{p}\right)\right]\quad\text{for every }x\in[0,1].
\end{equation}

Put
\begin{equation*}
\FPS=\left\{\phi\in\FP\mid\phi\text{ satisfies }\cref{FE}\right\}.
\end{equation*}
By \cref{prop:FunctionalEquation} each $\phi\in\FPS$ determines exactly one $\mu\in\FMS$ and each $\mu\in\FMS$ is determined by some $\phi\in\FPS$. Therefore, we see that for describing $\FMS$, it suffices to describe $\FPS$. Clearly, any $\phi\in\FPS$ can be decomposed in a canonical way into the absolutely continuous part, the continuous and singular part, and the jump part, i.e.\ the function defined by $\phi_j(x)=\sum_{y\leq x}[\lim_{z\to y+}\phi(z)-\phi(y)]$ for every $x\in[0,1]$  (see~\cite[page~58]{LebesgueNew}, cf.~\cite[page~139]{AmbrosioFuscoPallara2000}).
Note that by a jump function we  understand any $\phi\in\FPS$ for which the measure $\mu\in\FMS$ that corresponds to $\phi$ by~\cref{mu-phi} is atomic (or discrete), i.e.\ there are a countable (finite or infinite) set $\{x_j\mid j\in J\}\subset[0,1)$ and a sequence $(\alpha_j)_{j\in J}$ with $\sum_{j\in J}\alpha_j=1$ such that $\mu=\sum_{j\in J}\alpha_j\delta_{x_j}$. 

We say that the transformation $\Phi\in\FF$ is \emph{nonsingular} if $\Phi^{-1}(B)$ is of measure zero whenever $B\in\Borel$ is of measure zero (cf.~\cite[Definition~3.2.2]{LasotaMackey1994}). Note that each $p$\nobreakdash-adic transformation is nonsingular. 

\begin{prop}\label{prop:ASJ}
Assume that $S$ is nonsingular and sends sets of measure zero to sets of measure zero. If $\phi\in\FPS$, then its absolutely continuous, continuous and singular, and jump parts satisfy \cref{FE}.	
\end{prop}

\begin{proof}
Fix $\phi\in\FPS$. By the canonical decomposition (see \cite[Chapter 16, Section~E]{Jones1993}, cf.~\cite[Theorem~7.1.45]{KannanKrugerKing1996}) there exist exactly one (non-decreasing) absolutely continuous function $\phi_a$, exactly one (non-decreasing) continuous and singular function $\phi_s$, and exactly one (non-decreasing) jump function $\phi_j$ such that $\phi_a(0)=\phi_s(0)=0$ and
\begin{align*}
\phi_a(x)+\phi_s(x)+\phi_j(x)&=\phi(x)=\sum_{k=1}^{p}\left[\phi(f_k^{-1}(x))-\phi(f_k^{-1}(0))\right]\\
&=\sum_{k=1}^{p}\phi_a(f_k^{-1}(x))+\sum_{k=1}^{p}\phi_s(f_k^{-1}(x))\\
&\quad+\sum_{k=1}^{p}\phi_j(f_k^{-1}(x))-\sum_{k=1}^{p}\phi(f_k^{-1}(0))
\end{align*} 
for every $x\in[0,1]$.
To see that $\sum_{k=1}^{p}\phi_a\circ f_k^{-1}$ is absolutely continuous, one can apply the Banach--Zarecki theorem (see \cite[Theorem~7.14]{BrucknerBrucknerThomson} or \cite[Theorem 7.1.38]{KannanKrugerKing1996}).
It is easy to check that the function $\sum_{k=1}^{p}\phi_s\circ f_k^{-1}$ is continuously singular and $\sum_{k=1}^{p}\phi_j\circ f_k^{-1}$ is a jump function. Then from the uniqueness in the canonical decomposition there exists a constant $c\in\mathbb R$ such that for every $x\in[0,1]$ we have 
\begin{equation*}
\phi_a(x)=\sum_{k=1}^{p}\phi_a(f_k^{-1}(x))-\sum_{k=1}^{p}\phi(f_k^{-1}(0))+c
\end{equation*}
and since
$0=\phi_a(0)=\sum_{k=1}^{p}\phi_a(f_k^{-1}(0))-\sum_{k=1}^{p}\phi(f_k^{-1}(0))+c$
we see that $\phi_a$ satisfies \cref{FE} and 
\begin{equation*}
\phi_s(x)+\phi_j(x)=\sum_{k=1}^{p}\phi_s(f_k^{-1}(x))
+\sum_{k=1}^{p}\phi_j(f_k^{-1}(x))-c.
\end{equation*}
Again by the  canonical decomposition there exists a constant $d\in\mathbb R$ such that for every $x\in[0,1]$ we have
\begin{equation*}
\phi_s(x)=\sum_{k=1}^{p}\phi_s(f_k^{-1}(x))-c+d
\end{equation*}
and since $0=\phi_s(0)=\sum_{k=1}^{p}\phi_s(f_k^{-1}(0))-c+d$
we see that $\phi_s$ satisfies \cref{FE} and
\begin{equation*}
\phi_j(x)=\sum_{k=1}^{p}\phi_j(f_k^{-1}(x))-d.
\end{equation*}
Finally, since $0=\phi(0)=\phi_a(0)+\phi_s(0)+\phi_j(0)=\phi_j(0)$ we conclude that $0=\phi_j(0)=\sum_{k=1}^{p}\phi_j(f_k^{-1}(0))-d$, and hence $\phi_j$ satisfies \cref{FE}.
\end{proof}

From \Cref{prop:ASJ} we see that for describing the family $\FPS$ it suffices to describe each of its three subfamilies:
\begin{align*}
\FPS^\textnormal{a}&=\{\phi\in\FPS\mid\phi\text{ is an absolutely continuous function}\},\\
\FPS^\textnormal{s}&=\{\phi\in\FPS\mid\phi\text{ is a continuous and singular function}\},\\
\FPS^\textnormal{j}&=\{\phi\in\FPS\mid\phi\text{ is a jump function}\}.\\
\end{align*}


\renewcommand{\theequation}{4.\arabic{equation}}\setcounter{equation}{0}
\section{The family \texorpdfstring{$\mathcal P_S$}{Pₛ}}\label{S4}

Although we know that it is sufficient to describe the families $\FPS^\textnormal{a}$, $\FPS^\textnormal{s}$, and $\FPS^\textnormal{j}$, we will start with a general description of the family $\FPS$. 

In this section, the main tool are Banach limits.
Recall that a Banach limit $B\colon l^\infty\to\mathbb R$ is a linear and continuous extension of the functional $\lim\colon c \to\mathbb R$, where $l^\infty$ is the space of all bounded sequences of real numbers, and $c$ is its subspace consisting of all convergent sequences. In other words, a Banach limit is a linear, positive, shift invariant and normalized functional defined on~$l^\infty$. 
The reader interested in Banach limits can consult, e.g., \cite{AlekhnoSemenovSukochevUsachev2018, SemenovSukochevUsachev2019, SemenovSukochevUsachev2020, Sofi2021, DasNanda2022, SemenovSukochevUsachev2023} and the references therein.

From now on, we fix a Banach limit $\B$.

Define $\T_S\colon \FI\to\FI$ putting
\begin{equation*}
\T_S\phi(x)=\sum_{k=1}^{p}\left[\phi(f_k^{-1}(x))-\phi(f_k^{-1}(0))\right]
\quad\text{for every }x\in[0,1].
\end{equation*}
Obviously,
\begin{equation}\label{TSpsi}
\T_S\phi=\phi\quad\text{for every }\phi\in\FPS.
\end{equation}

With any $\phi\in\FI$ we associate the function $\B_S^\phi\colon[0,1]\to[0,1]$ defined by
\begin{equation*}
\B_S^\phi(x)=\B((\T_S^m\phi(x))_{m\in\mathbb N}).
\end{equation*} 
Clearly, $\B_S^\phi\in\FI$. However, the function $\B_S^\phi$ may not be left-continuous. Moreover, it can happen that $\lim_{x\to 1-}\B_S^\phi(x)\in[0,1)$. 
To get the left-continuity, we need to modify the just defined function $\B_S^\phi$. Therefore, we define the function $\BB_S^\phi\colon[0,1]\to[0,1]$ putting
\begin{equation*}
\BB_S^\phi=\B_S^\phi\quad\hbox{in the case }\lim_{x\to 1-}\B_S^\phi(x)=0
\end{equation*}	
and 
\begin{equation*}
\BB_S^\phi(x)=\begin{cases*}
0,&if $x=0$,\\
\frac{1}{\alpha}\displaystyle{\lim_{z\to x-}}\B_S^\phi(z),& if $x\in(0,1]$\\
\end{cases*}
\quad\text{in the case }\alpha=\lim_{x\to 1-}\B_S^\phi(x)\in(0,1].
\end{equation*}

Denote by $\chi_A$ the characteristic function of the set $A\subset[0,1]$ 
restricted to $[0,1]$.

\begin{remark}\label{rem:BB}\
\begin{itemize}
\item[\textrm{(i)}] If $\lim_{x\to 1-}\B_S^\phi(x)=0$, then  $\BB_S^\phi=\chi_{\{1\}}$. In particular, $\BB_S^{\chi_{\{1\}}}=\chi_{\{1\}}$.
\item[\textrm{(ii)}] If $\lim_{x\to 1-}\B_S^\phi(x)\in(0,1]$, then $\BB_S^\phi\in\FP$.
\end{itemize}
\end{remark}


Note that $\chi_{\{1\}}$ does not correspond to an invariant measure in our setting.
The reason is that we are considering $[0,1)$.
If we extended $S$ to~$[0,1]$ by setting $S(1)=1$, then $\chi_{\{1\}}$ would correspond to an invariant measure: $\delta_1$.
However, for this extension of~$S$, we would then need to replace the left-continuity with the right-continuity and modify the definition of $\BB_S^\phi$.

\begin{theorem}\label{thm:BB}
We have $\{\BB_S^\phi\mid\phi\in\mathcal I\}=\mathcal P_S\cup\{\chi_{\{1\}}\}$.
\end{theorem}

\begin{proof}
($\supset$) 
Since $\chi_{\{1\}}\in\mathcal I$ and $\BB_S^{\chi_{\{1\}}}=\chi_{\{1\}}$, we have $\chi_{\{1\}}\in\{\BB_S^\phi\mid\phi\in\mathcal I\}$.
Fix now $\phi\in\mathcal P_S$ and $x\in[0,1]$. Repeatedly applying~\cref{TSpsi}, we get
\begin{equation*}
\B_S^\phi(x)=\B((T_S^m\phi(x))_{m\in\mathbb N})=\B((\phi(x))_{m\in\mathbb N})=\phi(x).
\end{equation*} 
Hence $\BB_S^\phi=\phi$, which yields $\mathcal P_S\subset\{\BB^\phi_S\mid\phi\in\mathcal I\}$.
	
($\subset$) Fix $\phi\in\mathcal I$ and $x\in[0,1]$. Then
\begin{align*}
\B_S^\phi(x)&=\B\left(\left(\T_S^m\phi(x)\right)_{m\in\mathbb N}\right)
=\B\left(\left(\T_S^{m+1}\phi(x)\right)_{m\in\mathbb N}\right)\\
&=\B\left(\left(\sum_{k=1}^{p}\left[\T_S^m\phi(f_k^{-1}(x))
-\T_S^m\phi(f_k^{-1}(0))\right]\right)_{m\in\mathbb N}\right)\\
&=\sum_{k=1}^{p}\left[\B\left(\left(\T_S^m\phi(f_k^{-1}(x))\right)_{m\in\mathbb N}\right)
-\B\left(\left(\T_S^m\phi(f_k^{-1}(0))\right)_{m\in\mathbb N}\right)\right]\\
&=\sum_{k=1}^{p}\left[\B_S^\phi(f_k^{-1}(x))-\B_S^\phi(f_k^{-1}(0))\right].
\end{align*}
In consequence, according to \cref{rem:Continuity}, we get
\begin{equation*}
\BB_S^\phi=\sum_{k=1}^{p}\left[\BB_S^\phi\circ f_k^{-1}-\BB_S^\phi(f_k^{-1}(0))\right],
\end{equation*}
which jointly with \Cref{rem:BB} yields 
$\{\BB^\phi_S\mid\phi\in\FI\}\subset\FPS\cup\{\chi_{\{1\}}\}$.
\end{proof}

Before showing how \Cref{thm:BB} works in the $p$-adic case, let us note that 
\begin{equation*}
\T_{S_p}\phi(x)=\sum_{k=0}^{p-1}\left[\phi\left(\frac{x+k}{p}\right)
-\phi\left(\frac{k}{p}\right)\right]\quad\text{for every }x\in[0,1]
\end{equation*}
and, by induction,
\begin{equation}\label{Tpm}
\T_{S_p}^m\phi(x)=\sum_{k=0}^{p^m-1}\left[\phi\left(\frac{x+k}{p^m}\right)
-\phi\left(\frac{k}{p^m}\right)\right]\quad\text{for all $m\in\mathbb N$ and $x\in[0,1]$}.
\end{equation}

\begin{example}\label{ex:x^2}
Fix $\phi\in\mathcal I$ of the form $\phi(x)=x^2$. By \cref{thm:BB} we have $\BB_{S_p}^\phi\in\mathcal P_{S_p}\cup\{\chi_{\{1\}}\}$. Fix $x\in[0,1]$. Then using \cref{Tpm} we obtain
\begin{align*}
\B_{S_p}^\phi(x)&=\B((T_{S_p}^m\phi(x))_{m\in\mathbb N})
=\B\left(\left(\sum_{k=0}^{p^m-1}\left[\left(\frac{x+k}{p^m}\right)^2
-\left(\frac{k}{p^m}\right)^2\right]\right)_{m\in\mathbb N}\right)\\
&=\B\left(\left(x+\frac{x^2-x}{p^m}\right)_{m\in\mathbb N}\right)
=\lim_{m\to\infty}\left(x+\frac{x^2-x}{p^m}\right)=x.
\end{align*}
Therefore, $\BB_{S_p}^\phi=\id_{[0,1]}$ and the invariant measure that corresponds to $\id_{[0,1]}$ by~\cref{mu-phi} is the one-dimensional Lebesgue measure on $[0,1]$.
\end{example}


\renewcommand{\theequation}{5.\arabic{equation}}\setcounter{equation}{0}
\section{The family \texorpdfstring{$\mathcal P_{S_p}^\textnormal{a}$}{Pₛᵃ}}\label{S5}

It is well-known (see \cite{Renyi1957, Rohlin1961}) that 
\begin{equation}\label{PSpa}
\mathcal P_{S_p}^\textnormal{a}=\{\id_{[0,1]}\},
\end{equation}
however it can be deduced from the following more general result about the transformation~$S$ in~\cite{LasotaMackey1994}, which can be stated as follows.

\begin{theorem}[{see \cite[Theorem 6.2.1]{LasotaMackey1994}; cf.~\cite[Theorem 1]{LasotaYork1973}}]\label{thm:621}
Assume that $f_k\in C^2([a_k,a_{k+1}))$ for every $k\in\{1,\ldots,p\}$ and there exist $\alpha\in(1,\infty)$ and $\beta\in\mathbb R$ such that $\alpha\leq f_k'(x)$ and $-f_k''(x)\leq\beta [f_k'(x)]^2$ for all $k\in\{1,\ldots,p\}$ and $x\in[a_k,a_{k+1})$, then $\mathcal P_S^\textnormal{a}$ consists of exactly one function.
\end{theorem}

Since $\id_{[0,1]}$ is the unique absolutely continuous function belonging to the family $\mathcal P_{S_p}$, in view of \Cref{thm:BB}, it would be interesting (see \Cref{ex:x^2}) to find the biggest (in the sense of inclusion) subfamily of the family $\mathcal I$ such that for any function $\phi$ from this subfamily we have $\BB_{S_p}^\phi=\id_{[0,1]}$. We do not know what this subfamily is, but we have the following result.

\begin{theorem}\label{thm:PSpa=id} 
If $\phi\in\mathcal I$ is absolutely continuous, then $\lim_{m\to\infty}\T_{S_p}^m\phi(x)=x$ for every $x\in[0,1]$. In particular, $\BB_{S_p}^\phi=\B_{S_p}^\phi=\id_{[0,1]}$.
\end{theorem}

\begin{proof}
Fix an absolutely continuous function $\phi\in\FI$ and let $\Phi\colon\mathbb R\to [0,1]$ be the $1$-periodic extension of $\phi|_{[0,1)}$. According to \cite[Theorem 7.4.4]{Lojasiewicz1988} for every $m\in\mathbb N$ the function 
\begin{equation*}
R_m(\phi')=\frac{1}{p^m}\sum_{k=0}^{p^m-1}\Phi'\left(\cdot+\frac{k}{p^m}\right)
\end{equation*} 
belongs to $L^1([0,1])$, and moreover, 
\begin{equation}\label{RiemannSum}
\lim_{m\to\infty}\left\lVert R_m(\phi')-\int_{0}^{1}\phi'(s)ds\right\rVert_1=0;
\end{equation}
see e.g.~\cite[Section 2]{RuchWeber2006}.

Fix $x\in[0,1]$ and $m\in\mathbb N$. 
Applying again~\cite[Theorem 7.4.4]{Lojasiewicz1988}, \cref{Tpm}, and the periodicity of $\Phi$, we obtain 
\begin{align*}
|\T_{S_p}^m\phi(x)-x|&=\left|\int_{0}^x \left[(\T_{S_p}^m\phi)'(t)-1\right]dt\right|
\leq\int_{0}^{x}|(\T_{S_p}^m\phi)'(t)-1|\,dt\\
&=\int_{0}^{x}\left|\frac{1}{p^m}\sum_{k=0}^{p^m-1}\phi'\left(\frac{t+k}{p^m}\right)
-\int_{0}^{1}\phi'(s)\,ds\right|dt\\
&\leq\int_{0}^{1}\left|\frac{1}{p^m}\sum_{k=0}^{p^m-1}\Phi'\left(\frac{t+k}{p^m}\right)
-\int_{0}^{1}\phi'(s)\,ds\right|dt\\
&=\sum_{i=0}^{p^m-1}\frac{1}{p^m}\int_{0}^{1}\left|\frac{1}{p^m}
\sum_{k=0}^{p^m-1}\Phi'\left(\frac{t+i+k}{p^m}\right)-\int_{0}^{1}\phi'(s)\,ds\right|dt\\
&=\sum_{i=0}^{p^m-1}\int_{\frac{i}{p^m}}^{\frac{i+1}{p^m}}\left|\frac{1}{p^m}
\sum_{k=0}^{p^m-1}\Phi'\left(z+\frac{k}{p^m}\right)-\int_{0}^{1}\phi'(s)\,ds\right|dz\\
&=\int_{0}^{1}\left|\frac{1}{p^m}\sum_{k=0}^{p^m-1}\phi'\left(z+\frac{k}{p^m}\right)
-\int_{0}^{1}\phi'(s)\,ds\right|dz\\
&=\lVert R_m(\phi')-\int_{0}^{1}\phi'(s)ds\rVert_1.
\end{align*}
This jointly with \cref{RiemannSum} completes the proof.
\end{proof}


\renewcommand{\theequation}{6.\arabic{equation}}\setcounter{equation}{0}
\section{The family \texorpdfstring{$\mathcal{P}_{S_p}^\textnormal{j}$}{Pₛʲ}}\label{S6}

For $\phi\in\FI$ we set $\phi(x+)=\lim_{y\to x+}\phi(y)$. The next remark can be easily verified with the use of \cref{rem:Continuity}.

\begin{remark}\label{rem:lim+}
Assume that $\phi\in\FPS$ and $\phi(0+)=0$.
Then the function $\psi\colon[0,1]\to[0,1]$ defined by $\psi(x)=\phi(x+)-\phi(x)$ satisfies~\cref{FESp} and is such that $\psi(x)=0$ if and only if $\phi$ is continuous at $x$.
\end{remark}

Note that if $\phi\in \mathcal{P}_{S_p}$ fails to be continuous at a point~$x$, then there are further points~$\frac{x+k}{p}$, where it fails to be continuous as well, and the jumps have to add up.
But now, all of these points generate themselves further points of discontinuity.
However, this process cannot continue forever as the sum of all possible jumps is bounded from above by~$1$.
Therefore, points have to start to coincide, forcing them to be of a certain form.
This is the motivation behind the definition of the following set:
\begin{equation*}
D:=\bigcup_{m\in\mathbb N}\left\{\frac{i}{p^m-1}\mid i\in\{0,\ldots,p^m-2\}\right\}.
\end{equation*}

\begin{prop}\label{prop:ContinuouityPoints}
Every $\phi\in\FPSp$ is continuous at each point outside the set $D$.
\end{prop}

\begin{proof}
Fix $\phi\in\FPSp$ and $z_0\in [0,1]\setminus D$. 

If $\phi(0+)=1$, then $\phi$ is constant on $(0,1]$, and hence continuous at~$z_0$. Thus we assume that $\phi(0+)\in[0,1)$. We also assume that $\phi(0+)=0$, otherwise we replace $\phi$ with $\Phi\colon[0,1]\to[0,1]$ given by
\begin{equation*}
\Phi(x)=\begin{cases*}
0,&if $x=0$,\\
\frac{\phi(x)-\phi(0+)}{1-\phi(0+)},&if $x\in(0,1]$.
\end{cases*}
\end{equation*}
Here, we gloss over a small technical issue, namely that constant functions differing from~$0$ are not contained in~$\FPSp$.

Put 
\begin{equation*}
\alpha=\phi(z_0+)-\phi(z_0)
\end{equation*} 
and note that $\alpha\geq 0$ as $\phi$ is non-decreasing. For every $m\in\mathbb N$ we set
\begin{equation*}
C_m=\left\{\frac{z_0+i}{p^{m-1}}\mid i\in\{0,\ldots,p^{m-1}-1\}\right\}.
\end{equation*}
We want to show that
\begin{equation}\label{Cm}
\sum_{x\in C_m}[\phi(x+)-\phi(x)]=\alpha
\end{equation}
for every $m\in\mathbb N$. For $m=1$, this is clear by the definition of the number $\alpha$. Fix $m\in\mathbb N$ and assume that \cref{Cm} holds. By \cref{rem:lim+}, we have
\begin{equation*}
\alpha=\sum_{x\in C_m}\sum_{k=0}^{p-1}\left[\phi\left(\frac{x+k}{p}
+\right)-\phi\left(\frac{x+k}{p}\right)\right]
=\sum_{x\in C_{m+1}}[\phi(x+)-\phi(x)].
\end{equation*}
	
Since $z_0\notin D$, we have that $C_n\cap C_m=\emptyset$ for all $n,m\in\mathbb N$ with $n\neq m$. This jointly with \cref{Cm} and the fact that $\phi$ is non-decreasing gives
\begin{equation*}
\alpha M=\sum_{m=1}^M\sum_{x\in C_m}[\phi(x+)-\phi(x)]\leq \sum_{x\in[0,1]}[\phi(x+)-\phi(x)]\leq 1
\end{equation*}
for every $M\in\mathbb N$. Thus $\alpha=0$, and by \Cref{rem:lim+} we see that $\phi$ is continuous at $z_0$.
\end{proof}

In principle, if $f$ is not continuous at a point $x$, then there is at least one~$k$ such that $f$ also fails to be continuous at the point~$\frac{x+k}{p}$.
It turns out that there is exactly one such~$k$.
\begin{lemma}\label{lem:D}
Let $x\in [0,1]$. Then $\card\left(D\cap\{\frac{x+k}{p}\mid k\in\{0,\ldots,p-1\}\}\right)\leq 1$ and equality holds if and only if $x\in D$.
\end{lemma}

\begin{proof}
Fix $x\in[0,1]$ and assume by contradiction that there exist $k,l\in\{0,\ldots,p-1\}$ such that $\frac{x+k}{p},\frac{x+l}{p}\in D$, i.e.\ there are $m_1,m_2\in\mathbb N$, $i_1\in\{0,\ldots,p^{m_1}-2\}$, and $i_2\in\{0,\ldots,p^{m_2}-2\}$ such that $\frac{x+k}{p}=\frac{i_1}{p^{m_1}-1}$ and
$\frac{x+l}{p}=\frac{i_2}{p^{m_2}-1}$.
Then
\begin{equation*}
\frac{k-l}{p}=\frac{i_1(p^{m_2}-1)-i_2(p^{m_1}-1)}{(p^{m_1}-1)(p^{m_2}-1)}.
\end{equation*}
Since $\gcd(p,p^{m_1}-1)=\gcd(p,p^{m_2}-1)=1$, we conclude that $k=l$.

To prove the second part of the lemma assume first that $x\in D$. Then there exist $m\in\mathbb N$, $k\in\{0,\ldots,p-1\}$, and $l\in\{0,\ldots,p^{m-1}-1\}$ such that $x=\frac{k+lp}{p^m-1}$; note that $1\not\in D$ yields $k\neq p-1$ or $l\neq p^{m-1}-1$. Then
\begin{equation*}
\frac{x+k}{p}=\frac{k+lp +k(p^m-1)}{p(p^m-1)}=\frac{l+kp^{m-1}}{p^m-1}\in D.
\end{equation*}
Conversely, if there exists $k\in\{0,\ldots,p-1\}$ such that $\frac{x+k}{p}\in D$, then an easy calculation shows that $x\in D$.
\end{proof}

The next observation is an obvious consequence of the definition of $D$, the fact that $1\notin D$, and \cref{lem:D}.
Thus we omit the proof.

\begin{lemma}\label{lem:Dml}
If $x_0\in D$, then there exist $m\in\mathbb N$ and $(k_0,\ldots,k_{m-1})\in\{0,\ldots,p-1\}^m$ such that $k_0+k_1p+\cdots+k_{m-1}p^{m-1}\neq p^m-1$ and
\begin{equation}\label{formulaDml}
\frac{x_0+k_0+k_1p+\cdots+k_{m-1}p^{m-1}}{p^{m}}=x_0;
\end{equation}
moreover, the formulas
\begin{equation}\label{cycle}
x_{(n+1)(\hspace{-1.7ex}\mod m)}=\frac{x_n+k_n}{p}\quad\text{for every }n\in\{0,\ldots,m-1\}
\end{equation}
define the cycle $(x_0,\ldots,x_{m-1})$ with
\begin{equation*}\label{formulaxn}
x_n=\frac{\sum_{j=0}^{m-1}k_jp^{j+m-n}(\hspace{-2ex}\mod p^{m}-1)}{p^m-1}\in D\quad\text{for every } n\in\{1,\ldots,m-1\}.
\end{equation*}
\end{lemma}

Assume that there exists $\phi\in\FPSp$ which is discontinuous at $x_0\in D$. By \Cref{lem:D}, there are uniquely determined sequences $(x_n)_{n\in\mathbb N_0}$ (of real numbers from $[0,1]$) and $(k_n)_{n\in\mathbb N_0}$ (of integer numbers from $\{0,\ldots,p-1\}$) such that $x_{n+1}=\frac{x_n+k_n}{p}\in D$ for every $n\in\mathbb N_0$. \Cref{lem:Dml} says that the sequence $(x_n)_{n\in\mathbb N_0}$ is periodic, i.e.\ there exists $m\in\mathbb N$ such that $x_m=x_0$ and $x_{m+k}=x_k$ for all $k\in \field{N}$. Let $m\in\mathbb N$ be the smallest number with that property. Then the set $\{x_0,\ldots,x_{m-1}\}$ has $m$ distinct points and if $\phi\in\FPSp$ is discontinuous at a point of the set $\{x_0,\ldots,x_{m-1}\}$, then it is discontinuous at every point of that set, by the periodicity of $(x_n)_{n\in\mathbb N_0}$. This allows us to decompose the set $D$ into disjoint subsets $D^m_l$, where each of these sets is minimal (in the sense of inclusion) with the following two properties: 
\begin{enumerate}
\item[\textrm{(i)}] $\card(D^m_l)=m$ (the number $m\in\mathbb N$ will be called the \emph{level} of $D^m_l$, whereas the index $l\in\mathbb N_0$ is used to number the sets of level $m$);
\item[\textrm{(ii)}] If $\phi\in\FPSp$ is discontinuous at a point of $D^m_l$, then it is discontinuous at every point of $D^m_l$.
\end{enumerate}
In fact, each of the sets from the decomposition of $D$ is of the form described above. More precisely, fix a sequence $(k_0,\ldots,k_{m-1})\in\{0,\ldots,p-1\}^{m}$ such that $l=k_0+k_1p+\cdots+k_{m-1}p^{m-1}\neq p^m-1$ and choose $x_0\in D$ satisfying \cref{formulaDml}. Then the formula \cref{cycle} gives raise to the definition of
the cycle $(x_0,\ldots,x_{m-1})$ with
\begin{equation*}
x_n=\frac{lp^{m-n}(\hspace{-2ex}\mod p^{m}-1)}{p^m-1}
\quad\text{for every }n\in\{0,\ldots,m-1\},
\end{equation*}
and $\{x_0,\ldots,x_{m-1}\}\subset D$.

If $x_n\neq x_0$ for every $n\in\{0,\ldots,m-1\}$, then the cycle $(x_0,\ldots,x_{m-1})$ leads to the set $D^m_l=\{x_0,\ldots,x_{m-1}\}$ and since $(x_0,\ldots,x_{m-1})$ is a cycle, we have 
\begin{equation*}
D^m_l=D^m_{lp^n(\hspace{-2ex}\mod p^{m}-1)}\quad\text{for every }n\in\{1,\ldots,m-1\}.
\end{equation*}

Assume now that $j=\min\{n\in\{0,\ldots,m-1\}\mid x_n=x_0\}<m$. Then $j=\card(\{x_0,\ldots,x_{m-1}\})$ and $l=lp^{j}(\hspace{-1ex}\mod p^{m}-1)$, which means that $k_{j+n}=k_n$ for every $n\in\{0,\ldots,m-1-j\}$. Since $(x_0,\ldots,x_{m-1})$ is a cycle, we see that $m=jq$ with some $q\in\mathbb N$. Hence
\begin{align*}
l&=k_0+\cdots+k_jp^{j-1}+k_0p^j+\cdots+k_jp^{2j-1}+\cdots+k_0p^{j(q-1)}+\cdots+k_jp^{jq-1}\\
&=(k_0+\cdots+k_jp^{j-1})(1+p^j+\cdots+p^{(q-1)j}),
\end{align*}
and putting $l_0=k_0+\cdots+k_jp^{j-1}$ we obtain
\begin{align*}
\frac{l}{p^m-1}&=\frac{l_0(1+p^j+\cdots+p^{(q-1)j})}{p^{jq}-1}=
\frac{l_0(1+p^j+\cdots+p^{(q-1)j})}{(p^j-1)(1+p^j+\cdots+p^{(q-1)j})}\\
&=\frac{l_0}{p^j-1}.
\end{align*}
This leads to the set $D^j_{l_0}=\{x_0,\ldots,x_{j-1}\}$, and as previously, we have
\begin{equation*}
D^j_{l_0}=D^j_{l_0p^n(\hspace{-2ex}\mod p^{j})}\quad\text{for every }n\in\{1,\ldots,j-1\}.
\end{equation*}

Summarizing, we see that for each $m\in\mathbb N$ the sequence $(k_0,\ldots,k_{m-1})\in\{0,\ldots,p-1\}^{m}$ with $l=k_0+k_1p+\cdots+k_{m-1}p^{m-1}\neq p^m-1$ leads either to the set $D^m_l$ or to the set $D^j_{l_0}$ with $l_0=k_0+k_1p+\cdots+k_{j-1}p^{j-1}$ and $j$ being a divisor of $m$.
In particular, we have
\begin{itemize}
\item[(1)] $p-1$ sets of level $1$ of the form  
\begin{equation*}
D^1_l=\left\{\frac{l}{p-1}\right\},
\end{equation*}
where $l\in\mathbb L_1=\{0,\ldots,p-2\}$ (cf.~\cref{rem:FixedPoints});
\item[(2)] $\frac{p(p-1)}{2}$ sets of level $2$ of the form
\begin{equation*}
D^2_{k_0+k_1p}=D^2_{k_1+k_0p}=\left\{\frac{k_0+k_1p}{p^2-1},\frac{k_1+k_0p}{p^2-1}\right\},
\end{equation*}
where $k_0,k_1\in\{0,\ldots,p-1\}$ and $k_0\neq k_1$;
\item[(3)] $\frac{p(p^2-1)}{3}$ sets of level $3$ of the form 
\begin{align*}
D^3_{k_0+k_1p+k_2p^2}&=D^3_{k_2+k_0p+k_1p^2}=D^3_{k_1+k_2p+k_0p^2}\\
&=\left\{\frac{k_0+k_1p+k_2p^2}{p^3-1},
\frac{k_1+k_2p+k_0p^2}{p^3-1},\frac{k_2+k_0p+k_1p^2}{p^3-1}\right\},
\end{align*}
where $k_0,k_1,k_2\in\{0,\ldots,p-1\}$ and 
$k_0\neq k_1$ or $k_0\neq k_2$ or $k_1\neq k_2$;
\item[($m$)] $\card \mathbb L_m$ sets of level $m\geq 2$ of the form
\begin{align*}
D^m_l&=D^m_{lp(\hspace{-2ex}\mod p^{m}-1)}=\cdots=D^m_{lp^{m-1}(\hspace{-2ex}\mod p^{m}-1)}\\
&=\left\{\frac{l}{p^m-1},\frac{lp\hspace{.5ex}(\hspace{-2ex}\mod p^{m}-1)}{p^m-1},\ldots,\frac{lp^{m-1}(\hspace{-2ex}\mod p^{m}-1)}{p^m-1}\right\},
\end{align*}
where
\begin{equation*}
l\in\mathbb L_m=\left\{\sum_{i=0}^{m-1}k_ip^{i}\mid (k_0,\ldots,k_{m-1})\in\{0,\ldots,p-1\}^m\text{ is acyclic}\right\}.
\end{equation*}
\end{itemize} 

Now, with each of the set $D^m_l$ we want to associate a jump function belonging to $\FPSp$ that is discontinuous exactly at the points of the set $D^m_l$. The first lemma concerns sets of level~$1$, and its proof is a simple verification.

\begin{lemma}\label{lem:jumpFunction1}
For every $l\in\{0,\ldots,p-2\}$ the function $\phi^1_l\colon[0,1]\to[0,1]$ given by
\begin{equation*}
\phi^1_l(x)=\chi_{(\frac{l}{p-1},1]}(x)
\end{equation*}
belongs to $\FPSp^\textnormal{j}$ and is discontinuous only at the point of the set $D^1_l$.
\end{lemma}

The next remark is an immediate consequence of the description of the sets $D^m_l$; cf.~\cref{lem:D}.

\begin{remark}\label{rem:yi}
Let $x\in D^m_l$. Then:
\begin{enumerate}[label=(\roman*)]
\item\label{xy} there exist unique  $y\in D^m_l$ and $k\in\{0,\ldots,p-1\}$ such that $x=\frac{y+k}{p}$;
\item\label{zx} there exist unique $z\in D^m_l$ and $q\in\{0,\ldots,p-1\}$ such that $z=\frac{x+q}{p}$.
\end{enumerate}
\end{remark}

\begin{remark}\label{rem:unique}
Let $x\in(0,1)$ and $y\in D^m_l$. Then there exists at most one $k\in\{0,\ldots,p-1\}$ such that $\frac{k}{p}\leq y<\frac{x+k}{p}$.
\end{remark}

\begin{proof}
Assume by contradiction that $\frac{k}{p}\leq y<\frac{x+k}{p}$ and $\frac{q}{p}\leq y<\frac{x+q}{p}$ with $0\leq k<q\leq p-1$, then $\frac{q}{p}\leq y<\frac{x+k}{p}<\frac{1+k}{p}\leq\frac{q}{p}$, a contradiction.
\end{proof}

Before formulating the second lemma, let us recall that $1\notin D$ and since $D^1_0=\{0\}$, we have $0\notin D^m_l$ for any $l\in\mathbb L_m$, whenever $m\geq 2$. 

\begin{lemma}\label{lem:jumpFunction2}	
Assume that $D^m_l=\{y_0,\ldots,y_{m-1}\}$, where $m\geq 2$, $l=\sum_{i=0}^{m-1}k_ip^{i}$ with $(k_0,\ldots,k_{m-1})\in\{0,\ldots,p-1\}^m$, and $0<y_0<\cdots<y_{m-1}<1$. Then the function $\phi^m_l\colon[0,1]\to[0,1]$ given by
\begin{equation*}
\phi^m_l(x)=\frac{1}{m}\sum_{i=0}^{m-1}\chi_{(y_i,1]}(x)
\end{equation*}
belongs to $\FPSp^\textnormal{j}$ and is discontinuous exactly at points of the set $D^m_l$.
\end{lemma}

\begin{proof}
It is clear from the definition of $\phi^m_l$ that it is non-decreasing, left-continuous, and discontinuous only at points of the set $D^m_l$. 
Moreover, $\phi^m_l(0)=0$ and $\phi^m_l(1)=1$. So, to complete the proof it suffices to show that $\phi^m_l$ satisfies \cref{FESp}, which means that
\begin{equation}\label{FEchi}
\phi^m_l(x)=\frac{1}{m}\sum_{i=0}^{m-1}\sum_{k=0}^{p-1}
\left[\chi_{(y_i,1]}\left(\frac{x+k}{p}\right)-\chi_{(y_i,1]}\left(\frac{k}{p}\right)\right]
\end{equation}
for every $x\in[0,1]$.

It is easy to check that \Cref{FESp} holds for $x\in\{0,1\}$. Fix $x\in(0,1)$. We will consider three cases:
\begin{itemize}
\item[(a)] $x\in(0,y_0]$,
\item[(b)] $x\in(y_{j-1},y_j]$ for some $j\in\{1,\ldots,m-1\}$,
\item[(c)] $x\in(y_{m-1},1)$.
\end{itemize}

Case (a). In this case we have $\phi^m_l(x)=0$.
For proving that \cref{FEchi} holds, it suffices to show that for no $i\in\{0,\ldots,m-1\}$ and $k\in \{0,\ldots,p-1\}$ we have
\begin{equation}\label{yi}	
\frac{k}{p}\leq y_i<\frac{x+k}{p};
\end{equation}
indeed, then for all $i\in\{0,\ldots,m-1\}$ and $k\in \{0,\ldots,p-1\}$ we have
\begin{equation*}
\chi_{(y_i,1]}\left(\frac{x+k}{p}\right)=\chi_{(y_i,1]}\left(\frac{k}{p}\right).
\end{equation*}

Let us assume by contradiction that there are $i\in\{0,\ldots,m-1\}$ and $k\in\{0,\ldots,p-1\}$ satisfying \cref{yi}. By \Cref{rem:yi}~\cref{xy} there exists $y_j\in D^m_l$ and $k_0\in \{0,\ldots,p-1\}$  such that $y_i=\frac{y_j+k}{p}$.
Actually, $k$~and~$k_0$ are the same:
\begin{equation*}
\frac{k}{p}\leq \frac{y_j+k_0}{p} <\frac{x+k}{p},               
\end{equation*}
yielding
\begin{equation*}
-1<y_j-x <k-k_0\leq y_j<1,
\end{equation*}
and thus $k=k_0$.
Then
\begin{equation*}
\frac{k}{p}\leq\frac{y_j+k}{p}<\frac{x+k}{p}\leq\frac{y_0+k}{p},
\end{equation*}
and hence $y_j\in(0,y_0)$, a contradiction to the fact that $y_0$ is minimal. 

Case (b). In this case we have $\phi^m_l(x)=\frac{j}{m}$. 
To prove that \cref{FEchi} holds, it is enough to show that there are exactly $j$ distinct points $y_{i_1},\ldots,y_{i_j}\in D^m_l$ such that for each $i\in\{i_1,\ldots,i_j\}$ there exists only one $k\in\{0,\ldots,p-1\}$ satisfying $\cref{yi}$; indeed, then 
\begin{equation*}
\sum_{k=0}^{p-1}\left[\chi_{(y_i,1]}\left(\frac{x+k}{p}\right)
-\chi_{(y_i,1]}\left(\frac{k}{p}\right)\right]=
\begin{cases*}
1,&if $i\in\{i_1,\ldots,i_j\}$,\\
0,& if $i\notin\{i_1,\ldots,i_j\}.$
\end{cases*}
\end{equation*} 

Assume first that $y_i$ satisfies \cref{yi} with some $k\in\{0,\ldots,p-1\}$. Note that $k$ is unique by \Cref{rem:unique}. From \Cref{rem:yi}~\cref{xy} there exists $y_q\in D^m_l$ and $k_0\in \{1,\ldots,p-1\}$ such that $y_i=\frac{y_q+k_0}{p}$.
As
\begin{equation*}
  \frac{k_0}{p}\leq \frac{y_q+k_0}{p}
               =y_i
               <\frac{k_0+1}{p},
\end{equation*}
the uniqueness just mentioned now guarantees that $k=k_0$.

Then, by~\cref{yi}, we have
\begin{equation*}
\frac{k}{p}\leq\frac{y_q+k}{p}<\frac{x+k}{p}\leq\frac{y_j+k}{p},
\end{equation*}
and hence $y_q\in \{y_0,\ldots,y_{j-1}\}$. This shows that there exist at most $j$ distinct points in $D^m_l$ with the required property. Fix now $y_q\in \{y_0,\ldots,y_{j-1}\}$.  By \Cref{rem:yi}~\cref{zx} there are $y_i\in D^m_l$ and $k\in \{0,\ldots,p-1\}$ such that $y_i=\frac{y_q+k}{p}$, and hence
\begin{equation*}
\frac{k}{p}<y_i=\frac{y_q+k}{p}\leq\frac{y_{j-1}+k}{p}<\frac{x+k}{p},
\end{equation*}
which means that \cref{yi} holds. Therefore, we showed that each of the elements in $\{y_0,\ldots,y_{j-1}\}$ leads to some $y_i\in D^m_l$ and $k\in\{0,\ldots,p-1\}$ satisfying~\cref{yi}.
Actually more is true: the $y_i$ are distinct, and in consequence, the required property holds.

Case (c). In this case $\phi_l^m(x)=1$.
Let $y_i\in\{y_0,\ldots,y_{m-1}\}$. From \Cref{rem:yi}~\cref{xy} there are unique $y_q\in D^m_l$ and $k\in \{0,\ldots,p-1\}$ such that $y_i=\frac{y_q+k}{p}$. Then 
\begin{equation*}
\frac{k}{p}<y_i=\frac{y_q+k}{p}\leq\frac{y_{m-1}+k}{p}<\frac{x+k}{p}
\end{equation*} 
and, by \Cref{rem:unique}, we see that for each $i\in \{0,\ldots,m-1\}$ there is exactly one $k\in\{0,\ldots,p-1\}$ satisfying \cref{yi}. Therefore, for every $i\in \{0,\ldots,m-1\}$ we have
\begin{equation*}
\sum_{k=0}^{p-1}\left[\chi_{(y_i,1]}\left(\frac{x+k}{p}\right)
-\chi_{(y_i,1]}\left(\frac{k}{p}\right)\right]=1,
\end{equation*} 
which implies that \cref{FEchi} holds.
\end{proof}

We conclude our investigations with the following result.

\begin{theorem}\label{thm:jamp}
Let $\phi\in\FPSp$. Then $\phi\in\FPSp^\textnormal{j}$ if and only if there exists a set $\{\alpha^m_l\mid m\in\mathbb N, l\in \mathbb L_m\}$ of non-negative real numbers with $\sum_{m\in\mathbb N}\sum_{l\in\mathbb L_m}\alpha^m_l=1$ such that
\begin{equation}\label{jumpSolution}
\phi=\sum_{m\in\mathbb N}\sum_{l\in\mathbb L_m}\alpha^m_l\phi^m_l.
\end{equation}
\end{theorem}

\begin{proof}
Making use of~\cref{lem:jumpFunction1,lem:jumpFunction2} it is easy to check that the function $\phi$ given by~\cref{jumpSolution} belongs to $\FPSp^\textnormal{j}$.

Let now $\phi\in\FPSp^\textnormal{j}$. Choose $x_0\in D$ with $s_1:=\phi(x_0+)-\phi(x_0)\in(0,1)$. Then there are $m_1\in\mathbb N$ and $l_1\in\mathbb L_{m_1}$ such that $x_0\in D^{m_1}_{l_1}$. By \Cref{lem:D} for every $x\in D^{m_1}_{l_1}$ we have $\phi(x+)-\phi(x)=s_1$. Moreover, $m_1s_1\in \opencl{0}{1}$. Therefore, the function $\phi-m_1s_1\phi^{m_1}_{l_1}$ is non-decreasing, continuous at every point of the set $D^{m_1}_{l_1}$, and $\phi(1)-m_1s_1\phi^{m_1}_{l_1}(1)=1-m_1s_1$. If $m_1s_1=1$, then $\phi=m_1s_1\phi^{m_1}_{l_1}$, and we are done. If $m_1s_1\in(0,1)$, then  $\frac{1}{1-m_1s_1}(\phi-m_1s_1\phi^{m_1}_{l_1})\in\FPSp^\textnormal{j}$ and repeating the same arguments as before we can choose $m_2\in\mathbb N$, $l_2\in\mathbb L_{m_2}$, and $s_2\in(0,1)$ such that the function $\frac{1}{1-m_1s_1}(\phi-m_1s_1\phi^{m_1}_{l_1})-m_2s_2\phi^{m_2}_{l_2}$ is non-decreasing, continuous at every point of the set $D^{m_1}_{l_1}\cup D^{m_2}_{l_2}$, and $\frac{1}{1-m_1s_1}(\phi(1)-m_1s_1\phi^{m_1}_{l_1}(1))-m_2s_2\phi^{m_2}_{l_2}(1)=1-m_2s_2$. Again, if $m_2s_2=1$, then $\phi=m_1s_1\phi^{m_1}_{l_1}+m_2s_2\phi^{m_2}_{l_2}$ and we are done. Otherwise we continue the reasoning. Since the set of all discontinuity points of $\phi$ is countable (see \Cref{prop:ContinuouityPoints}), the proof can be completed by induction.
\end{proof}

From \cref{thm:jamp} we see that there are solutions of equation \cref{FESp} that are discontinuous exactly at the points of the set $D$.


\renewcommand{\theequation}{7\arabic{equation}}\setcounter{equation}{0}
\section{The family \texorpdfstring{$\mathcal P_{S_p}^\textnormal{s}$}{Pₛˢ}}\label{S7}

One can ask if it would be possible to formulate any counterpart of \Cref{thm:PSpa=id} for the family $\FPSp^\textnormal{s}$. Unfortunately, we do not know if, for a continuous and singular $\phi\in\FI$, the sequence $(\T^m_{S_p}\phi)_{m\in\mathbb N}$ converges pointwise, and if so, to which function. We do not even know if, for a continuous (and singular) $\phi\in\FI$, the function $\B_{S_p}^\phi$ is continuous (and singular).
So, we cannot determine members of the family $\FPSp^\textnormal{s}$ by any formula similar to that in~\cref{thm:PSpa=id}.

Fix $\phi\in\FI$. By \cref{thm:BB} we have $\BB_{S_p}^{\phi}\in\FPSp$ or $\BB_{S_p}^{\phi}=\chi_{\{1\}}$. If $\BB_{S_p}^{\phi}\neq\chi_{\{1\}}$, then applying \cref{prop:ASJ} together with \cref{PSpa} and \cref{thm:jamp} we conclude that there are constants $\alpha\in[0,1]$ and $\alpha^m_l\in[0,1]$, for all $m\in\mathbb N$ and $l\in \mathbb L_m$, such that $\alpha+\sum_{m\in\mathbb N}\sum_{l\in\mathbb L_m}\alpha^m_l\in[0,1]$ and
\begin{equation*}
\BB_{S_p}^{\phi}-\alpha\id_{[0,1]}-\sum_{m\in\mathbb N}\sum_{l\in\mathbb L_m}\alpha^m_l\phi^m_l\in\FPSp^\textnormal{s}.
\end{equation*}
All the constants can be easily calculated. As $\BB_{S_p}^{\phi}$ is differentiable almost everywhere, we have (almost everywhere)
\begin{equation*}
\alpha=(\BB_{S_p}^{\phi})',
\end{equation*}
whereas for all $m\in\mathbb N$ and $l\in \mathbb L_m$ we have
\begin{equation*}
\alpha^m_l=m\left[\BB_{S_p}^{\phi}\left(\frac{l}{p^m-1}+\right)-\BB_{S_p}^{\phi}\left(\frac{l}{p^m-1}\right)\right].
\end{equation*}
In consequence, for any $\phi\in\FI$ such that $\BB_{S_p}^{\phi}\neq\chi_{\{1\}}$ the formula
\begin{equation*}
\BB_{S_p}^{\phi}-(\BB_{S_p}^{\phi})'\id_{[0,1]}-\sum_{m\in\mathbb N}m\sum_{l\in\mathbb L_m}\left[\BB_{S_p}^{\phi}\left(\frac{l}{p^m-1}+\right)-\BB_{S_p}^{\phi}\left(\frac{l}{p^m-1}\right)\right]\phi^m_l
\end{equation*}
defines a function belonging to $\FPSp^\textnormal{s}$. Note that $\phi\in\FPSp^\textnormal{s}$ gives $\BB_{S_p}^{\phi}=\phi$, so the above formula holds as $(\BB_{S_p}^{\phi})'=0$ and
$\BB_{S_p}^{\phi}(\frac{l}{p^m-1}+)=\BB_{S_p}^{\phi}(\frac{l}{p^m-1})$ for all $m\in\mathbb N$ and $l\in \mathbb L_m$.

Let us finish this short section with some information on previous results concerning the family $\FPSdwa^\textnormal{s}$. Namely, in~\cite{MorawiecZurcher2018} (see also \cite[Section 5C]{Kairies1997}), it was observed that $\FPSdwa^\textnormal{s}$ contains a large family of strictly increasing and H\"{o}lder continuous functions that are convex combinations of the singular de~Rham functions from \cite{Rham1956}  
(studied earlier in \cite{Cesaro}, \cite{Faber}, and \cite{Salem}). Next in~\cite{MorawiecZurcher2021} more new families of strictly increasing functions belonging to $\FPSdwa^\textnormal{s}$ were found. Among them a family of functions that are not H\"{o}lder continuous. Recently, in~\cite{MorawiecZurcher2024}, it was observed that $\FPSdwa^\textnormal{s}$ also contains a quite large family of Cantor-type functions. All the results from the mentioned papers concerns the family $\FPSdwa^\textnormal{s}$, however some of them can be generalized to the family $\FPSp^\textnormal{s}$.


\section*{Acknowledgement}
The research activities were co-financed by the funds granted under the Research Excellence Initiative of the University of Silesia in Katowice and by the University of Silesia Institute of Mathematics (Iterative Functional Equations and Real Analysis program).

%
%
%


\end{document}